\documentclass[11pt]{amsart}
\usepackage[T1]{fontenc}
\usepackage[utf8]{inputenc}
\usepackage{lmodern}
\usepackage{microtype}
\usepackage{amsmath,amssymb,amsthm,mathtools}
\usepackage{tikz}
\usetikzlibrary{arrows.meta,positioning,fit,shapes.geometric}
\usepackage{booktabs}
\usepackage{array}
\usepackage{enumitem}
\usepackage{adjustbox}
\usepackage{geometry}
\usepackage{hyperref}
\hypersetup{
  colorlinks=true,
  linkcolor=blue,
  citecolor=blue,
  urlcolor=blue,
  pdftitle={Local Complexity, Overlap Control, and Global Simplicity in Partition Graphs: A Conceptual Synthesis for Gn and Kn=Cl(Gn)},
  pdfauthor={Fedor B. Lyudogovskiy},
  pdfsubject={2020 Mathematics Subject Classification: Primary 05A17, 55U10; Secondary 05C69, 05C75, 05C76, 05C30, 06A07, 55P15},
  pdfkeywords={integer partitions, partition graph, clique complex, local morphology, local simplex dimension, canonical cover, nerve lemma, intersection poset, overlap control, Euler characteristic, bouquet of spheres},
  bookmarks=true,
  bookmarksnumbered=true
}

\newtheorem{theorem}{Theorem}[section]
\newtheorem{synthesistheorem}[theorem]{Synthesis Theorem}
\newtheorem{backgroundtheorem}[theorem]{Background Theorem}
\newtheorem{proposition}[theorem]{Proposition}
\newtheorem{backgroundproposition}[theorem]{Background Proposition}
\newtheorem{corollary}[theorem]{Corollary}
\newtheorem{lemma}[theorem]{Lemma}
\newtheorem{openproblem}[theorem]{Open Problem}

\theoremstyle{definition}

\newtheorem{synthesisprinciple}[theorem]{Synthesis Principle}

\theoremstyle{remark}
\newtheorem{remark}[theorem]{Remark}

\DeclareMathOperator{\Cl}{Cl}
\DeclareMathOperator{\Par}{Par}

\title[Local Complexity and Global Simplicity in Partition Graphs]{Local Complexity, Overlap Control, and Global Simplicity in Partition Graphs:\\
A Conceptual Synthesis for \texorpdfstring{\(G_n\)}{Gn} and \texorpdfstring{\(K_n=\Cl(G_n)\)}{Kn=Cl(Gn)}}

\author[F. B. Lyudogovskiy]{Fedor B. Lyudogovskiy}
\date{}



\begin{document}

\begin{abstract}
For each positive integer \(n\), let \(G_n\) be the partition graph whose vertices are the partitions of \(n\), with adjacency defined by an elementary transfer of one unit between parts, followed by reordering, and let
\[
K_n=\Cl(G_n)
\]
be its clique complex. Previous work revealed two apparently contrasting features. Locally, the graphs \(G_n\) become increasingly rich: local clique dimensions grow, degree landscapes refine, support jumps and simplex layers proliferate, and axial, rear, central, shell, and directional structures become more pronounced. Globally, however, the clique complexes remain homotopically simple:
\[
K_n\simeq \bigvee^{b_n}S^2,
\qquad
b_n=\chi(K_n)-1.
\]

This paper gives a conceptual synthesis explaining why these facts are compatible. The local side is governed by ordered local transfer types, which determine local neighborhood graphs, degrees, local clique numbers, and local simplex dimensions. The global side is governed not by the largest local simplices, but by the overlap pattern of canonical full star/top simplices. These simplices form a good cover \(\mathcal C_n\), giving
\[
K_n\simeq N(\mathcal C_n)=N_n.
\]
The nerve admits an anchor-cover and intersection-poset reduction
\[
N_n\simeq \Delta(J_n),
\qquad
\dim\Delta(J_n)\le 2.
\]
High-dimensional local simplices therefore occur inside contractible containers, while global topology is controlled by a low-dimensional overlap poset. Within this family, the qualitative topological problem reduces to the numerical computation and interpretation of a single integer, the Euler characteristic \(\chi(K_n)\), equivalently the bouquet rank \(b_n=\chi(K_n)-1\).
\end{abstract}

\maketitle

\tableofcontents

\section{Introduction}

The partition graph \(G_n\) is the graph whose vertices are the partitions of \(n\), with two partitions adjacent when one can be obtained from the other by an elementary transfer of one unit between parts, followed by reordering. Its clique complex will be denoted by
\[
K_n=\Cl(G_n).
\]
The general notation and terminology follow the preceding papers in the partition-graph cycle, especially the topological and local-morphological foundations \cite{LyudogovskiyHomotopyPartitionGraph,LyudogovskiyLocalMorphology}.

This paper addresses a structural contrast that has emerged across the study of the family \(G_n\). As \(n\) grows, the graphs \(G_n\) develop increasingly rich local and morphological structure. Their degree landscapes become more varied, their local clique structure becomes thicker, support strata and support jumps become more intricate, simplex layers and phase boundaries become more pronounced, and axial, rear, central, and directional phenomena become increasingly visible \cite{LyudogovskiyGrowingDiscreteGeometricObject,LyudogovskiyAxialMorphology,LyudogovskiyDegreeLandscape,LyudogovskiyBoundaryRear,LyudogovskiyDirectionalGeometry,LyudogovskiyMorphogenesis,LyudogovskiySimplicialShells,LyudogovskiySimplexLayers,LyudogovskiySupportJumps,LyudogovskiyJumpGradient}.

At the same time, the global homotopy type of the clique complex remains strikingly simple:
\[
K_n\simeq \bigvee^{b_n}S^2,
\qquad
b_n=\chi(K_n)-1.
\]
Here and below, an empty wedge is understood as a point. Thus the qualitative global topology is completely controlled by the Euler characteristic. In particular,
\[
\widetilde H_i(K_n)=0\qquad(i\ne 2),
\]
and the only remaining global homological rank is
\[
\operatorname{rank}\widetilde H_2(K_n)=b_n.
\]

This paper explains why these two phenomena are compatible by identifying the structural mechanism that separates local and morphological growth from global homotopy complexity.

\subsection{The local/global question}

The guiding question is: \emph{Why does local and morphological complexity grow while global topology stays simple?}

On the local side, the elementary transfer rule produces increasingly rich local neighborhoods. For a partition
\[
\lambda=(s_1^{m_1},\ldots,s_t^{m_t})\vdash n,
\qquad
s_1>\cdots>s_t>0,
\]
the local transfer behavior is controlled by the ordered local transfer type
\[
\mathfrak t(\lambda)
=
(t;\alpha_1,\ldots,\alpha_t;\beta_1,\ldots,\beta_t),
\]
where
\[
\alpha_i=\mathbf 1_{m_i=1},
\qquad
\beta_i=\mathbf 1_{s_i-s_{i+1}=1},
\qquad
s_{t+1}=0.
\]
As recalled from the local morphology paper \cite{LyudogovskiyLocalMorphology}, these binary data determine the local transfer graph \(B(\lambda)\), and hence the local neighborhood graph
\[
G_n[N_{G_n}(\lambda)]\cong L(B(\lambda)).
\]
Consequently, they determine local invariants such as degree, local clique number, and local simplex dimension. The distribution of these local transfer types across the partition set generates much of the observed morphology of \(G_n\).

On the global side, however, the topology of \(K_n\) is not governed directly by the largest local simplices. It is governed by the way canonical contractible pieces overlap. More precisely, every simplex of \(K_n\) is contained in a full star-simplex or a full top-simplex. These pieces form a canonical cover
\[
\mathcal C_n
\]
of \(K_n\). The cover is good, hence by the nerve lemma and the canonical-cover theorem \cite{BorsukNerve,BjornerTopologicalMethods,KozlovCombinatorialAlgebraicTopology,LyudogovskiyHomotopyPartitionGraph}
\[
K_n\simeq N(\mathcal C_n)=N_n.
\]
The nerve \(N_n\) admits a further anchor-cover and intersection-poset reduction
\[
N_n\simeq \Delta(J_n),
\]
where \(J_n\) is the intersection poset of anchor simplices. The key rigidity statement is that this poset has no long strict chains, so
\[
\dim\Delta(J_n)\le 2.
\]

This is the central mechanism of the paper:
\[
\boxed{\text{local complexity grows inside contractible containers,}}
\]
while
\[
\boxed{\text{global topology is governed by a low-dimensional overlap poset.}}
\]

\begin{figure}[ht]
\centering
\begin{tikzpicture}[
  node distance=1.25cm and 1.35cm,
  box/.style={draw, rounded corners, align=center, inner sep=6pt, minimum height=8mm},
  arrow/.style={-Latex, thick}
]
\node[box] (lt) {local transfer\\types};
\node[box, right=of lt] (lc) {local clique\\richness};
\node[box, right=of lc] (mor) {morphology\\of \(G_n\)};
\node[box, below=1.1cm of lc] (glob) {global type\\\(K_n\simeq \bigvee^{b_n}S^2\)};
\draw[arrow] (lt) -- (lc);
\draw[arrow] (lc) -- (mor);
\draw[arrow] (lc) -- node[right, font=\small, align=center] {contained in\\star/top containers} (glob);
\end{tikzpicture}
\caption{The basic local/global contrast. Local transfer-type diversity generates increasing local and morphological complexity in \(G_n\), while the clique complex \(K_n\) remains homotopy equivalent to a bouquet of two-spheres.}
\label{fig:basic-contrast}
\end{figure}
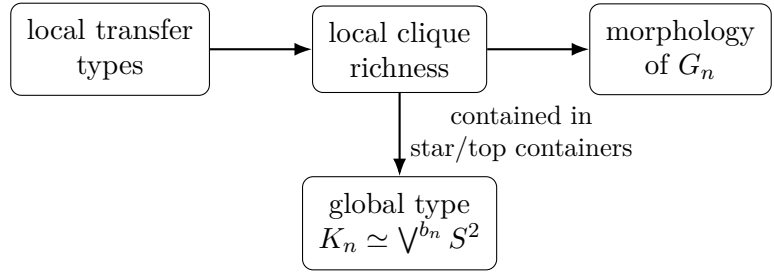

\subsection{Why large local simplices are not enough}

A naive expectation might be that the appearance of larger and larger simplices in \(K_n\) should lead to increasingly complicated global topology. This expectation is natural but incorrect.

Indeed, local simplex dimension is unbounded. For the staircase partition
\[
\delta_t=(t,t-1,\ldots,2,1)\vdash T_t,
\qquad
T_t=\frac{t(t+1)}2,
\]
the local simplex dimension is
\[
\dim_{\mathrm{loc}}(\delta_t)=t-1.
\]
Hence \(K_{T_t}\) contains simplices of dimension at least \(t-1\), and the simplicial dimension of \(K_n\) is unbounded along the family. Nevertheless,
\[
K_n\simeq \bigvee^{b_n}S^2
\]
for every \(n\). In particular,
\[
\widetilde H_k(K_n)=0\qquad(k\ge 3).
\]
Thus the implication
\[
\text{large local simplices}\Longrightarrow\text{high-dimensional global homology}
\]
fails for this family: a large simplex is still contractible, and high-dimensional homology does not appear. What matters globally is not only the size of individual simplices, but the pattern by which the canonical containers intersect.

\subsection{Logical status of statements}

Because this is a synthesis paper, the labels used below are part of the exposition. Results labelled \emph{Background Theorem} or \emph{Background Proposition} are imported from earlier papers in the cycle and are not reproved here, although they are occasionally accompanied by a short proof sketch recalling the mechanism. Results labelled \emph{Theorem}, \emph{Proposition}, or \emph{Corollary} are either short consequences of those background results or compact repackagings needed for the present synthesis. Statements labelled \emph{Synthesis Principle} are conceptual consequences or organizing principles; they are not intended to assert new classification theorems. Open problems are stated explicitly in Section~\ref{sec:open}. This convention is used to keep the distinction between proved results, strict consequences, synthesis, and open interpretation visible throughout the paper.

\subsection{A map of the main objects}\label{subsec:object-map}

The paper uses several complexes and posets at once. The following table fixes the levels of the construction. It is included to make clear which objects are graphs or complexes on partitions, which objects are covers by containers, and which objects record overlaps among those containers.

\begin{table}[ht]
\centering
\small
\begin{adjustbox}{max width=\textwidth}
\begin{tabular}{lll}
\toprule
Object & What its vertices are & Role in the synthesis \\
\midrule
\(G_n\) & partitions of \(n\) & graph-level transfer geometry \\
\(K_n=\Cl(G_n)\) & partitions of \(n\) & clique complex; raw simplicial object \\
\(\mathcal C_n\) & full star/top simplices of \(K_n\) & canonical contractible containers \\
\(N_n=N(\mathcal C_n)\) & containers \(U\in\mathcal C_n\) & first overlap complex of the cover \\
\(A_\lambda\subset N_n\) & containers through \(\lambda\) & anchor simplex attached to a partition \\
\(A_S\) & containers containing all vertices of \(S\) & anchor intersection \\
\(J_n\) & nonempty anchor intersections \(A_S\) & overlap poset with short chains \\
\(\Delta(J_n)\) & elements of \(J_n\) & low-dimensional homotopy model \\
\(b_n\) & --- & bouquet rank, equal to \(\chi(K_n)-1\) \\
\bottomrule
\end{tabular}
\end{adjustbox}
\caption{A reader's map of the main objects. The passage from \(K_n\) to \(N_n\) changes vertices from partitions to containers; the passage from \(N_n\) to \(J_n\) changes the focus from containers to their anchor intersections.}
\label{tab:object-map}
\end{table}

This table also indicates where information is lost and where it is reorganized. Passing from \(K_n\) to \(N_n\) forgets the internal simplicial dimension of each full star/top container, but keeps track of which containers have common partitions. The passage from \(N_n\) to \(J_n\) preserves the anchor-intersection structure and discards redundant descriptions of the same intersection. The low-dimensionality theorem concerns this last object, not the raw dimension of \(K_n\).

\subsection{Scope and organization}

This paper is not intended as a complete reproof of the local morphology theory or of the global bouquet theorem. Those results are used here as established inputs, especially the topology, local morphology, numerical topology, and morphology papers \cite{LyudogovskiyHomotopyPartitionGraph,LyudogovskiyLocalMorphology,LyudogovskiyGrowingDiscreteGeometricObject,LyudogovskiyNumericalTopology}. Nevertheless, the exposition is written so that the reader can follow the local/global mechanism without constantly returning to the preceding papers: the local transfer formalism, the canonical cover, the anchor construction, and the intersection-poset reduction are all recalled at the level needed for the synthesis.

The aim is narrower and more synthetic: to isolate the local/global contrast, identify the overlap-control mechanism, and clarify why the qualitative topological problem reduces to the numerical computation of \(\chi(K_n)\) and \(b_n\).

For orientation, Section~\ref{sec:basic} formulates the basic contrast; Section~\ref{sec:local} recalls the local transfer mechanism; Sections~\ref{sec:containers}--\ref{sec:overlap} develop the global overlap mechanism; Section~\ref{sec:compatible} assembles the synthesis; Section~\ref{sec:numerical} records numerical consequences; Section~\ref{sec:open} lists open conceptual problems; and Section~\ref{sec:conclusion} concludes.

\section{The Basic Contrast}\label{sec:basic}

The purpose of this section is to isolate the central contrast studied in the paper. On the one hand, the partition graph \(G_n\) develops increasingly rich local and morphological structure as \(n\) grows. On the other hand, the clique complex
\[
K_n=\Cl(G_n)
\]
has a globally simple homotopy type:
\[
K_n\simeq \bigvee^{b_n}S^2,
\qquad
b_n=\chi(K_n)-1.
\]
The coexistence of these two facts has a structural explanation.

\subsection{Local simplex dimension}

For a partition \(\lambda\vdash n\), define its local simplex dimension by
\[
\dim_{\mathrm{loc}}(\lambda)
=
\max\{\dim\sigma:\sigma\subset K_n,\ \lambda\in\sigma\}.
\]
Equivalently,
\[
\dim_{\mathrm{loc}}(\lambda)=\omega_{\mathrm{loc}}(\lambda)-1,
\]
where \(\omega_{\mathrm{loc}}(\lambda)\) is the maximum size of a clique of \(G_n\) containing \(\lambda\). This is a local invariant of the vertex \(\lambda\). It measures local simplicial thickness, not global homotopy type.

The local morphology theory gives an exact formula for \(\dim_{\mathrm{loc}}(\lambda)\) in terms of ordered local transfer type \cite{LyudogovskiyLocalMorphology}. We recall only the part needed here. Let
\[
\lambda=(s_1^{m_1},\ldots,s_t^{m_t})\vdash n,
\qquad
s_1>\cdots>s_t>0,
\qquad
s_{t+1}=0.
\]
Define
\[
\alpha_i=\mathbf 1_{m_i=1},
\qquad
\beta_i=\mathbf 1_{s_i-s_{i+1}=1}.
\]
Then
\[
\dim_{\mathrm{loc}}(\lambda)
=
\max
\left\{
\max_{1\le i\le t}(t+1-\alpha_i-\beta_i),
\max_{1\le j\le t+1}(t-\alpha_j-\beta_{j-1})
\right\},
\]
with the conventions
\[
\beta_0=0,
\qquad
\alpha_{t+1}=0.
\]
This formula is one of the basic bridges between local transfer combinatorics and the visible simplicial morphology of \(G_n\).

\subsection{Unbounded local simplex dimension}

\begin{proposition}[Staircase vertices have large local simplex dimension]\label{prop:staircase}
Let
\[
T_t=\frac{t(t+1)}2
\]
and let
\[
\delta_t=(t,t-1,\ldots,2,1)\vdash T_t
\]
be the staircase partition. Then
\[
\dim_{\mathrm{loc}}(\delta_t)=t-1.
\]
Hence \(K_{T_t}\) contains a simplex of dimension \(t-1\), so \(\dim K_{T_t}\ge t-1\). In particular, the simplicial dimension of \(K_n\) is unbounded along the sequence \(n=T_t\).
\end{proposition}

\begin{proof}
For the staircase partition \(\delta_t=(t,t-1,\ldots,1)\), the support size is \(t\). Every support block is a singleton block, and every consecutive support gap is equal to \(1\). Hence
\[
\alpha_i=1,
\qquad
\beta_i=1
\qquad(1\le i\le t).
\]
Substituting into the local simplex dimension formula gives
\[
\max_{1\le i\le t}(t+1-\alpha_i-\beta_i)=t-1.
\]
For the second family of terms, using \(\beta_0=0\) and \(\alpha_{t+1}=0\), one obtains
\[
t-\alpha_j-\beta_{j-1}\le t-1
\]
for every \(j\), with equality at the endpoints. Therefore
\[
\dim_{\mathrm{loc}}(\delta_t)=t-1.
\]
Since \(\dim_{\mathrm{loc}}(\delta_t)\) is the largest dimension of a simplex of \(K_{T_t}\) containing \(\delta_t\), the complex \(K_{T_t}\) contains a simplex of dimension \(t-1\). This proves the assertion.
\end{proof}

\subsection{Global homotopy remains two-dimensional in character}

We use the following established global result as background \cite{LyudogovskiyHomotopyPartitionGraph}.

\begin{backgroundtheorem}[Global simplicity of the clique complex]\label{thm:global-simplicity}
For every \(n\ge 1\),
\[
K_n\simeq \bigvee^{b_n}S^2,
\qquad
b_n=\chi(K_n)-1.
\]
In particular,
\[
\widetilde H_i(K_n)=0
\qquad
\text{for }i\ne 2,
\]
and
\[
\widetilde H_2(K_n)\cong \mathbb Z^{b_n}.
\]
Moreover, \(K_n\) admits a homotopy model of dimension at most \(2\).
\end{backgroundtheorem}

Combining Proposition~\ref{prop:staircase} with Background Theorem~\ref{thm:global-simplicity} gives the basic local/global contrast.

\begin{theorem}[Local high-dimensionality with global bouquet type]\label{thm:contrast}
The family of clique complexes \(K_n=\Cl(G_n)\) has unbounded simplicial dimension. Indeed,
\[
\dim K_{T_t}\ge t-1
\]
for the triangular values \(T_t=t(t+1)/2\). Nevertheless, for every \(n\),
\[
K_n\simeq \bigvee^{b_n}S^2,
\qquad
b_n=\chi(K_n)-1.
\]
Thus local simplicial dimension is unbounded, while the global homotopy type is always a bouquet of \(2\)-spheres.
\end{theorem}

\begin{proof}
The unboundedness of the simplicial dimension follows from Proposition~\ref{prop:staircase}. The bouquet statement is Background Theorem~\ref{thm:global-simplicity}.
\end{proof}

Theorem~\ref{thm:contrast} says that \(K_n\) need not be a two-dimensional simplicial complex. It may contain high-dimensional simplices. The theorem says instead that high-dimensional local simplices do not force high-dimensional global homology. In the present family,
\[
\widetilde H_k(K_n)=0\qquad(k\ge 3).
\]
The reason is structural. Large simplices occur inside canonical contractible pieces. The global homotopy type is controlled not by the dimensions of these pieces alone, but by the way they overlap.

\section{Local Mechanisms Behind Growing Morphology}\label{sec:local}

The preceding section isolated the basic contrast. We now examine the local side. The aim is not to reprove the local morphology theory \cite{LyudogovskiyLocalMorphology}, but to recall the mechanism by which local transfer data generate the increasingly rich visible structure of the partition graphs.

\begin{figure}[ht]
\centering
\begin{tikzpicture}[
  node distance=1.05cm and 0.9cm,
  box/.style={draw, rounded corners, align=center, inner sep=5pt, minimum height=8mm},
  arrow/.style={-Latex, thick}
]
\node[box] (type) {\(\mathfrak t(\lambda)\)};
\node[box, right=of type] (B) {\(B(\lambda)\)};
\node[box, right=of B] (L) {\(L(B(\lambda))\)};
\node[box, right=of L] (inv) {\(\deg,\omega_{\rm loc}\)\\ \(\dim_{\rm loc}\)};
\node[box, right=of inv] (morph) {morphology\\of \(G_n\)};
\draw[arrow] (type) -- (B);
\draw[arrow] (B) -- (L);
\draw[arrow] (L) -- (inv);
\draw[arrow] (inv) -- (morph);
\end{tikzpicture}
\caption{The local transfer mechanism. Ordered binary transfer data determine the bipartite transfer graph \(B(\lambda)\), hence the line-graph neighborhood \(L(B(\lambda))\), and therefore the local invariants that feed the morphology of \(G_n\).}
\label{fig:local-pipeline}
\end{figure}
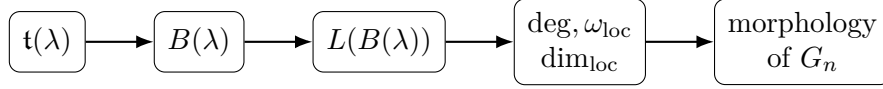

\subsection{Local transfer type and the local transfer graph}

Let
\[
\lambda=(s_1^{m_1},\ldots,s_t^{m_t})\vdash n,
\qquad
s_1>\cdots>s_t>0,
\qquad
s_{t+1}=0.
\]
The support size is \(\sigma(\lambda)=t\). A neighbor of \(\lambda\) is obtained by moving one unit from one part to another part, allowing the creation of a new part of size \(1\), and then reordering. The local neighborhood is governed by two kinds of obstructions: whether a support block is a singleton block, and whether two adjacent support values differ by \(1\). Thus
\[
\alpha_i=\mathbf 1_{m_i=1},
\qquad
\beta_i=\mathbf 1_{s_i-s_{i+1}=1}.
\]
The ordered local transfer type is
\[
\mathfrak t(\lambda)=(t;\alpha_1,\ldots,\alpha_t;\beta_1,\ldots,\beta_t).
\]

The local transfer theory can be expressed through a bipartite graph. Let the left side
\[
d_1,\ldots,d_t
\]
represent possible donor support blocks, and let the right side
\[
r_1,\ldots,r_{t+1}
\]
represent possible recipient levels, including the zero level from which a new part of size \(1\) may be created. Before local identifications, one obtains the complete bipartite graph \(K_{t,t+1}\). A nominal edge \(d_ir_j\) records a donor--recipient pair for an elementary transfer: one unit is removed from a part in the support block \(s_i\), inserted at the recipient level indexed by \(j\), and the resulting partition is reordered.

The passage from the nominal graph \(K_{t,t+1}\) to the actual local transfer graph has exactly two local deletion rules. If the donor block \(s_i\) consists of a single part, then the diagonal nominal edge \(d_ir_i\) is deleted. If the gap \(s_i-s_{i+1}\) is equal to \(1\), then the successor nominal edge \(d_ir_{i+1}\) is deleted. These are precisely the degeneracies recorded by \(\alpha_i\) and \(\beta_i\). Denote the resulting bipartite graph by \(B(\lambda)\).

The graph \(B(\lambda)\) is a compact local encoding of the elementary transfer rule near \(\lambda\). Its surviving edges are the admissible elementary transfers out of \(\lambda\). A key uniqueness observation is that two distinct admissible transfers from the same partition yield distinct neighboring partitions. Indeed, if
\[
\lambda(c\to a)=\lambda(d\to b),
\]
then in conjugate coordinates
\[
-e_{\operatorname{col}(c)}+e_{\operatorname{col}(a)}
=
-e_{\operatorname{col}(d)}+e_{\operatorname{col}(b)}.
\]
Admissibility excludes \(\operatorname{col}(c)=\operatorname{col}(a)\) and \(\operatorname{col}(d)=\operatorname{col}(b)\). Hence the negative and positive coordinates are uniquely determined, so
\[
\operatorname{col}(c)=\operatorname{col}(d),
\qquad
\operatorname{col}(a)=\operatorname{col}(b).
\]
Since removable and addable corners are uniquely determined by their column numbers, \(c=d\) and \(a=b\). This is Lemma~3.2 and Corollary~3.3 of the local morphology paper \cite{LyudogovskiyLocalMorphology}. Consequently, the surviving edges of \(B(\lambda)\) are in bijection with the neighbors of \(\lambda\), and compatibility among these transfers is governed by the line graph \(L(B(\lambda))\).

The following background proposition is quoted from \cite{LyudogovskiyLocalMorphology}.

\begin{backgroundproposition}[Local transfer type determines local clique geometry]\label{prop:local-transfer}
The ordered local transfer type \(\mathfrak t(\lambda)\) determines the local neighborhood graph \(G_n[N_{G_n}(\lambda)]\), the degree \(\deg(\lambda)\), the local clique number \(\omega_{\mathrm{loc}}(\lambda)\), and the local simplex dimension \(\dim_{\mathrm{loc}}(\lambda)\). More precisely,
\[
G_n[N_{G_n}(\lambda)]\cong L(B(\lambda)),
\]
where \(B(\lambda)\) is obtained from \(K_{t,t+1}\) by the deletion rules described above. Consequently,
\[
\deg(\lambda)=t(t+1)-\sum_{i=1}^t\alpha_i-\sum_{i=1}^t\beta_i,
\]
and
\[
\dim_{\mathrm{loc}}(\lambda)
=
\max
\left\{
\max_{1\le i\le t}(t+1-\alpha_i-\beta_i),
\max_{1\le j\le t+1}(t-\alpha_j-\beta_{j-1})
\right\},
\]
with \(\beta_0=0\) and \(\alpha_{t+1}=0\).
\end{backgroundproposition}

The formulas have a useful internal interpretation. The degree \(\deg(\lambda)\) is the number of surviving edges of \(B(\lambda)\). The two families of terms in the local simplex dimension formula are the row and column degrees of \(B(\lambda)\):
\[
\deg_{B(\lambda)}(d_i)=t+1-\alpha_i-\beta_i,
\]
and
\[
\deg_{B(\lambda)}(r_j)=t-\alpha_j-\beta_{j-1},
\]
with \(\beta_0=0\) and \(\alpha_{t+1}=0\). Since cliques in a line graph come from stars in the underlying graph, a large row or column star in \(B(\lambda)\) gives a large clique in the neighborhood of \(\lambda\). Adding \(\lambda\) itself gives the corresponding simplex of \(K_n\). In particular,
\[
\dim_{\mathrm{loc}}(\lambda)=\max\{\text{row and column degrees of }B(\lambda)\}.
\]

The importance of this proposition is that local complexity is reduced to finite binary data. The local neighborhood of a partition is not controlled by the full numerical size of its parts, but by the pattern \((t;\alpha;\beta)\). Thus many partitions with different part sizes may share the same local transfer type and hence the same local neighborhood geometry.

\subsection{A worked local example}\label{subsec:worked-local-example}

The formulas above are compact, but it is useful to see explicitly how the binary transfer data produce local geometry. Consider
\[
\lambda=(4,2,1)\vdash 7.
\]
Here the support values are
\[
s_1=4,\qquad s_2=2,\qquad s_3=1,
\]
so \(t=3\). All support blocks are singleton blocks, hence
\[
\alpha=(1,1,1).
\]
The gaps are
\[
s_1-s_2=2,\qquad s_2-s_3=1,\qquad s_3-s_4=1,
\]
where \(s_4=0\). Thus
\[
\beta=(0,1,1).
\]
The local transfer type is therefore
\[
\mathfrak t(4,2,1)=(3;1,1,1;0,1,1).
\]

Before degeneracies, the nominal local transfer graph is \(K_{3,4}\). The three singleton-block obstructions delete the diagonal edges
\[
d_1r_1,\qquad d_2r_2,\qquad d_3r_3,
\]
and the two unit-gap obstructions delete the successor edges
\[
d_2r_3,\qquad d_3r_4.
\]
The remaining seven edges form \(B(\lambda)\), shown schematically in Figure~\ref{fig:worked-local-B}.

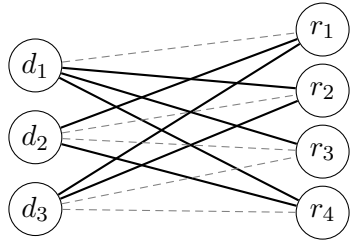
\begin{figure}[ht]
\centering
\begin{tikzpicture}[scale=0.95,
  donor/.style={circle,draw,inner sep=2pt,minimum size=7mm},
  rec/.style={circle,draw,inner sep=2pt,minimum size=7mm},
  live/.style={thick},
  dead/.style={densely dashed,gray}
]
\node[donor] (d1) at (0,2) {\(d_1\)};
\node[donor] (d2) at (0,1) {\(d_2\)};
\node[donor] (d3) at (0,0) {\(d_3\)};
\node[rec] (r1) at (4,2.5) {\(r_1\)};
\node[rec] (r2) at (4,1.65) {\(r_2\)};
\node[rec] (r3) at (4,0.8) {\(r_3\)};
\node[rec] (r4) at (4,-0.05) {\(r_4\)};

% live edges
\draw[live] (d1) -- (r2);
\draw[live] (d1) -- (r3);
\draw[live] (d1) -- (r4);
\draw[live] (d2) -- (r1);
\draw[live] (d2) -- (r4);
\draw[live] (d3) -- (r1);
\draw[live] (d3) -- (r2);

% deleted edges
\draw[dead] (d1) -- (r1);
\draw[dead] (d2) -- (r2);
\draw[dead] (d3) -- (r3);
\draw[dead] (d2) -- (r3);
\draw[dead] (d3) -- (r4);

\node[align=center] at (2,-0.85) {solid edges survive in \(B(\lambda)\); dashed edges are deleted by \(\alpha\) or \(\beta\)};
\end{tikzpicture}
\caption{The local transfer graph for \(\lambda=(4,2,1)\). The graph starts from \(K_{3,4}\). The deleted diagonal edges come from singleton support blocks, and the deleted successor edges come from unit gaps. The seven surviving edges give \(\deg(\lambda)=7\).}
\label{fig:worked-local-B}
\end{figure}

The degree formula gives
\[
\deg(4,2,1)=3\cdot4-(1+1+1)-(0+1+1)=7,
\]
which is exactly the number of surviving edges in Figure~\ref{fig:worked-local-B}.

For the local simplex dimension, the row terms are
\[
\begin{array}{c|ccc}
 i & 1 & 2 & 3\\
\hline
 t+1-\alpha_i-\beta_i & 3 & 2 & 2
\end{array}
\]
and the column terms, with \(\beta_0=0\) and \(\alpha_4=0\), are
\[
\begin{array}{c|cccc}
 j & 1 & 2 & 3 & 4\\
\hline
 t-\alpha_j-\beta_{j-1} & 2 & 2 & 1 & 2
\end{array}.
\]
Thus
\[
\dim_{\mathrm{loc}}(4,2,1)=3.
\]
The row \(d_1\) has three surviving incident transfer edges; in the line graph this gives a clique of three neighbors of \(\lambda\), and together with \(\lambda\) it gives a \(3\)-simplex in \(K_7\).

This example illustrates the general local mechanism for a small partition. The support size \(t=3\) creates the nominal scale \(t(t+1)=12\). The singleton and unit-gap data delete five local transfers, leaving degree \(7\). The largest surviving row or column star determines the local simplex dimension. The same calculation, performed across all partitions of \(n\), produces degree layers, simplex layers, and transition regions in \(G_n\).

\subsection{Degree landscape and simplex layers}

The degree formula, further analyzed in \cite{LyudogovskiyDegreeLandscape,LyudogovskiyDegreeTheory},
\[
\deg(\lambda)=t(t+1)-\sum_i\alpha_i-\sum_i\beta_i
\]
shows that degree is governed by support size and by two local obstruction counts: singleton support blocks and unit gaps. Thus the degree landscape is not random. It is stratified by support size and refined by the two binary obstruction patterns.

\begin{proposition}[Degree layers are transfer-type strata]\label{prop:degree-layers}
The degree function \(\deg:V(G_n)\to\mathbb Z\) is constant on ordered local transfer types. Consequently, degree layers in \(G_n\) are unions of transfer-type strata.
\end{proposition}

\begin{proof}
This follows immediately from the degree formula in Background Proposition~\ref{prop:local-transfer}.
\end{proof}

The same mechanism controls local simplex dimension. Large local simplices appear when there are rows or columns in \(B(\lambda)\) with many surviving incident edges. Equivalently, local thickness is produced by large families of elementary transfers that remain mutually compatible as cliques. For \(r\ge 0\), define the threshold simplex layer
\[
T_{\ge r}(n)=\{\lambda\vdash n:\dim_{\mathrm{loc}}(\lambda)\ge r\}.
\]

\begin{proposition}[Simplex layers are local transfer inequalities]\label{prop:simplex-layers}
For each \(r\ge 0\), the simplex layer \(T_{\ge r}(n)\) is determined by inequalities in the ordered local transfer type:
\[
\max
\left\{
\max_i(t+1-\alpha_i-\beta_i),
\max_j(t-\alpha_j-\beta_{j-1})
\right\}
\ge r.
\]
Thus the simplex-layer structure of \(G_n\) is induced by the distribution of local transfer types across \(\Par(n)\); this is the local form of the simplex-layer and phase-boundary framework developed in \cite{LyudogovskiySimplexStratification,LyudogovskiySimplexLayers}.
\end{proposition}

\begin{proof}
This is a direct restatement of the local simplex dimension formula.
\end{proof}

\subsection{Support jumps, anisotropy, and morphogenesis}

The support size \(\sigma(\lambda)=t\) is one of the most basic morphological coordinates on \(G_n\). It controls the leading term of degree and sets the scale for possible local clique dimensions. However, support size is not locally constant. An elementary transfer may preserve support size, increase support size by creating a new part, decrease support size by eliminating a singleton part, or change the support profile by merging with an adjacent value. Thus edges of \(G_n\) carry support jumps, in the sense studied in the support and jump-invariant papers \cite{LyudogovskiySupportJumps,LyudogovskiyJumpGradient}:
\[
\Delta_{\lambda\mu}\sigma=\sigma(\mu)-\sigma(\lambda),
\qquad
\lambda\mu\in E(G_n).
\]
These jumps mark transitions between local structural regimes.

More generally, let \(F:V(G_n)\to A\) be a vertex invariant, such as degree, support size, local simplex dimension, distance from the axis, distance from the spine, or distance from the boundary. The fibers \(F^{-1}(a)\) are morphological layers associated with \(F\). For a tuple \(\Phi_n=(F_1,\ldots,F_m)\), the fibers of \(\Phi_n\) give refined morphological strata, and the edgewise jump vector
\[
\Delta_{\lambda\mu}\Phi_n=\Phi_n(\mu)-\Phi_n(\lambda)
\]
records transitions between strata.

\begin{synthesisprinciple}[Morphology as fibers and jumps]\label{prin:fibers-jumps}
Many morphological structures studied in \(G_n\) can be organized as fibers and jump loci of vertex invariants. Degree layers are fibers of \(\deg\), simplex layers are upper fibers of \(\dim_{\mathrm{loc}}\), support strata are fibers of \(\sigma\), and phase boundaries are supported on edges where the relevant invariant changes.
\end{synthesisprinciple}

Directional anisotropy can be studied through non-uniform distributions of such edgewise jumps \cite{LyudogovskiyDirectionalGeometry}. Transfers may move toward or away from special regions such as the self-conjugate axis, the spine, the rear, the boundary, or the central zone. Thus the graph is not merely a collection of vertices with degrees and cliques; it has preferred transition patterns between regions.

As \(n\) grows, new support sizes, obstruction patterns, degree values, local simplex dimensions, support-jump patterns, and directional regimes appear \cite{LyudogovskiyMorphogenesis}. This is the morphogenetic side of the family \(G_1,G_2,\ldots\). The staircase example from Proposition~\ref{prop:staircase} is one clean witness: new values of \(n\) permit genuinely larger local clique configurations.

\begin{synthesisprinciple}[Morphogenesis from expanding transfer types]\label{prin:morphogenesis}
The observed morphological complexity of \(G_n\) is associated with the expanding range and distribution of ordered local transfer types as \(n\) grows. New support sizes and new obstruction patterns make new local neighborhood geometries available, and these in turn support new degree regimes, simplex layers, jump patterns, and regional structures. This principle does not assert monotonicity of any fixed numerical measure of complexity.
\end{synthesisprinciple}

\subsection{A local-to-morphology dictionary}\label{subsec:local-dictionary}

For reference, Table~\ref{tab:local-dictionary} summarizes the local dictionary used in this synthesis. The table is not a new classification theorem. Its purpose is to make explicit which local quantities are theorem-level consequences of the transfer-type formalism, and which larger morphological objects are obtained by distributing those quantities across the graph.

\begin{table}[ht]
\centering
\small
\begin{adjustbox}{max width=\textwidth}
\begin{tabular}{lll}
\toprule
Local or regional object & Immediate input & Morphological role \\
\midrule
support size \(\sigma(\lambda)=t\) & support of \(\lambda\) & scale for degree and local thickness \\
singleton data \(\alpha_i\) & multiplicities \(m_i\) & deletes diagonal local transfers \\
unit-gap data \(\beta_i\) & gaps \(s_i-s_{i+1}\) & deletes successor local transfers \\
transfer graph \(B(\lambda)\) & \((t;\alpha;\beta)\) & local compatibility pattern \\
neighborhood graph \(L(B(\lambda))\) & \(B(\lambda)\) & adjacency among neighbors of \(\lambda\) \\
degree \(\deg(\lambda)\) & number of surviving transfer edges & degree landscape \\
local simplex dimension & largest surviving row/column star & simplex layers and thickness \\
support jumps & edgewise change of \(\sigma\) & transitions between support regimes \\
jump vectors of invariants & edgewise changes of \(\Phi_n\) & phase boundaries and anisotropy \\
axis, spine, rear, center & regional/distance data & large-scale morphology of \(G_n\) \\
\bottomrule
\end{tabular}
\end{adjustbox}
\caption{A local-to-morphology dictionary. The first six rows are directly tied to the ordered local transfer-type mechanism. The later rows organize how these local data vary across \(G_n\) and interact with regional structures.}
\label{tab:local-dictionary}
\end{table}

The distinction in the last column is important. Local transfer type explains the geometry seen from a single vertex. Morphology is produced only after these local types and their associated invariants are distributed over the entire vertex set and compared along edges. Thus the local theory gives the alphabet, while the morphology of \(G_n\) is the grammar formed by arranging this alphabet across the graph.

This also clarifies the role of examples such as \((4,2,1)\). A single example illustrates how the local alphabet is computed, but it does not by itself explain a global region of \(G_n\). A degree layer, a simplex layer, a support-jump boundary, or an anisotropic regime appears only after the same local calculation is performed over many vertices and the resulting values are compared along the edges of the graph. This distinction clarifies a common point of confusion: local transfer formulas are pointwise theorems, whereas morphology emerges from their distribution across the graph.

\section{From Local Cliques to Global Containers}\label{sec:containers}

We now pass from local cliques to their global organization inside \(K_n\), using the star/top clique classification from the topological paper \cite{LyudogovskiyHomotopyPartitionGraph}. The first observation is that cliques in \(G_n\) are not arbitrary. Although large cliques occur, they occur inside highly structured canonical containers. These containers are the full star-simplices and full top-simplices.

The names are compatible with the local line-graph picture. Near a vertex \(\lambda\), cliques in
\[
G_n[N_{G_n}(\lambda)]\cong L(B(\lambda))
\]
come from row and column stars in the bipartite graph \(B(\lambda)\). A row star fixes a donor support block and varies the recipient level; a column star fixes a recipient level and varies the donor support block. The global star/top simplices are the canonical completions of these two local compatibility patterns inside the whole clique complex. This is why the local calculation of \(\dim_{\mathrm{loc}}\) and the global star/top cover are parts of the same mechanism rather than unrelated constructions.

\subsection{Full star/top simplices}

We recall the definition at the level needed here. Let \(\lambda\vdash n\). If \(c\) is a removable corner of \(\lambda\), define
\[
A_{\max}(\lambda,c)
=
\{a:\lambda(c\to a)\text{ is an admissible transfer}\}.
\]
The associated full star-simplex is
\[
\Sigma^{\mathrm{star}}_{\max}(\lambda,c)
=
\{\lambda\}\cup\{\lambda(c\to a):a\in A_{\max}(\lambda,c)\}.
\]
Thus a full star-simplex fixes the removable corner and allows all admissible addable corners. Dually, if \(a\) is an addable corner of \(\lambda\), define
\[
C_{\max}(\lambda,a)
=
\{c:\lambda(c\to a)\text{ is an admissible transfer}\},
\]
and the associated full top-simplex is
\[
\Sigma^{\mathrm{top}}_{\max}(\lambda,a)
=
\{\lambda\}\cup\{\lambda(c\to a):c\in C_{\max}(\lambda,a)\}.
\]
Thus a full top-simplex fixes the addable corner and allows all admissible removable corners. These are the full versions of the two local line-graph clique mechanisms: row stars and column stars in \(B(\lambda)\). Low-dimensional cases are allowed; in particular, a full star/top simplex may be only an edge. If two parameter choices give the same simplex, it is counted only once in the cover.

For each \(n\), let
\[
\mathcal C_n
\]
denote the collection of all distinct full star-simplices and full top-simplices in \(K_n\), as in the canonical-cover construction of~\cite[Definitions~2.3--2.5]{LyudogovskiyHomotopyPartitionGraph}. An element \(U\in\mathcal C_n\) is therefore a simplex of \(K_n\), hence contractible. It should be thought of as a canonical container for one coherent family of mutually compatible local transfers.

\begin{backgroundtheorem}[Star/top containment of cliques]\label{thm:star-top-containment}
Every clique of \(G_n\), equivalently every simplex of \(K_n\), is contained in a full star-simplex or in a full top-simplex. Equivalently,
\[
K_n=\bigcup_{U\in\mathcal C_n}U.
\]
Thus \(\mathcal C_n\) is a canonical cover of \(K_n\) by full star/top simplices.
\end{backgroundtheorem}

This theorem is one of the central combinatorial inputs behind the global topology. It says that even though \(K_n\) may contain many simplices of large dimension, these simplices are not arranged freely. They are contained in a controlled family of canonical contractible pieces.

\begin{proposition}[Large local simplices are contained in contractible canonical pieces]\label{prop:containers}
Let \(\sigma\subset K_n\) be any simplex. Then there exists \(U\in\mathcal C_n\) such that \(\sigma\subseteq U\). In particular, every local simplex of \(K_n\), regardless of its dimension, is contained in a contractible full star/top simplex.
\end{proposition}

\begin{proof}
The simplex \(\sigma\) is a clique in \(G_n\). By Background Theorem~\ref{thm:star-top-containment}, this clique is contained in a full star-simplex or a full top-simplex. This full star/top simplex is an element of \(\mathcal C_n\), and it is contractible because it is a simplex.
\end{proof}

\subsection{A schematic container picture}\label{subsec:container-picture}

The container viewpoint is easiest to read as a two-level description. At the first level, one sees many simplices of \(K_n\), possibly of large dimension. At the second level, each full star/top container is treated as one vertex of the nerve. Its internal dimension is forgotten by the nerve; what survives is whether it intersects other containers.

\begin{figure}[ht]
\centering
\begin{tikzpicture}[
  scale=1,
  container/.style={draw, rounded corners, minimum width=3.2cm, minimum height=1.55cm, align=center},
  small/.style={circle,fill=black,inner sep=1.25pt},
  nvert/.style={circle,draw,inner sep=2pt,minimum size=8mm},
  arrow/.style={-Latex, thick}
]
% upper left: one large container with internal dots
\node[container] (Ubox) at (0,0) {};
\node at (0,0.43) {large simplex};
\node at (0,0.12) {\(U\in\mathcal C_n\)};
\node[small] at (-0.9,-0.32) {};
\node[small] at (-0.45,-0.45) {};
\node[small] at (0.05,-0.30) {};
\node[small] at (0.55,-0.43) {};
\node[small] at (0.95,-0.25) {};
\node at (0,-1.05) {internal clique geometry};

% upper right: one nerve vertex
\node[nvert] (Nu) at (4.45,0) {\(U\)};
\node at (4.45,-1.05) {one nerve vertex};
\draw[arrow] (1.75,0) -- (3.85,0);

% lower left: two overlapping containers
\node[container, minimum width=2.45cm, minimum height=1.15cm] (U1box) at (-0.35,-3.0) {};
\node[container, minimum width=2.45cm, minimum height=1.15cm] (U2box) at (0.95,-3.0) {};
\node at (-0.8,-3.0) {\(U_1\)};
\node at (1.4,-3.0) {\(U_2\)};
\node at (0.3,-3.9) {overlap in \(K_n\)};

% lower right: edge in nerve
\node[nvert] (N1) at (4.0,-2.85) {\(U_1\)};
\node[nvert] (N2) at (5.45,-2.85) {\(U_2\)};
\draw[thick] (N1) -- (N2);
\node at (4.72,-3.9) {edge in \(N_n\)};
\draw[arrow] (2.25,-3.0) -- (3.5,-2.88);
\end{tikzpicture}
\caption{The container viewpoint. A full star/top simplex may be large, but it becomes a single vertex of the nerve. Global topology is governed by intersections among containers, not by their internal dimensions alone.}
\label{fig:container-picture}
\end{figure}
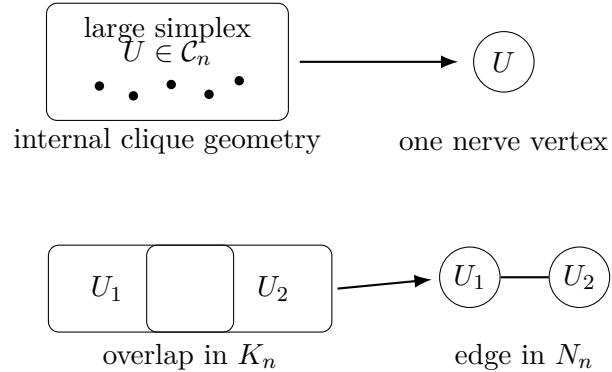

This schematic picture should not be confused with an additional construction. It is simply the nerve viewpoint applied to the canonical cover \(\mathcal C_n\). It explains why the next section concentrates on intersections among containers rather than on the raw list of all simplices of \(K_n\).

\subsection{Containers absorb local clique abundance}

A high-dimensional simplex in \(K_n\) may arise because a partition \(\lambda\) has many mutually compatible elementary transfers around it. But by Proposition~\ref{prop:containers}, this simplex lies inside some \(U\in\mathcal C_n\). Since \(U\) is contractible, a large clique contributes local thickness, but not by itself global homotopical complexity.

\begin{synthesisprinciple}[Local dimension is not the controlling global invariant]\label{prin:local-dim-not-global}
The simplicial dimension of \(K_n\) measures the size of the largest clique in \(G_n\). It is a genuine local-combinatorial quantity. However, the global homotopy type of \(K_n\) is controlled by the overlap pattern of the canonical cover \(\mathcal C_n\). Thus local simplex dimension contributes to the morphology of \(G_n\), but it is not the invariant that determines the qualitative homotopy type of \(K_n\).
\end{synthesisprinciple}

The cover \(\mathcal C_n\) is the interface between the local-morphological and global-topological sides of the theory. Locally, its elements are full star/top containers generated by compatible transfer families. Globally, their intersections define the nerve
\[
N_n=N(\mathcal C_n).
\]
Thus the passage
\[
K_n\quad\leadsto\quad \mathcal C_n\quad\leadsto\quad N_n
\]
replaces raw simplicial complexity by overlap combinatorics.

\section{The Overlap-Control Mechanism}\label{sec:overlap}

We now turn to the global side of the synthesis. Sections~\ref{sec:basic}--\ref{sec:containers} showed that local clique abundance is real, structured, and unbounded, but that every simplex is contained in a full star/top container. The global question is therefore not the size of individual simplices, but the way these canonical containers overlap.

This section recalls the overlap-control mechanism in a self-contained form. The detailed star/top classification and the original proofs are given in the topological paper \cite{LyudogovskiyHomotopyPartitionGraph}; here we include enough of the construction to make the present synthesis readable on its own. The mechanism has two reductions:
\[
K_n\simeq N_n
\]
by the canonical star/top cover, and
\[
K_n\simeq \Delta(J_n)
\]
by a closure-operator construction on the anchor-intersection poset. The second reduction is the one that produces the low-dimensional model.

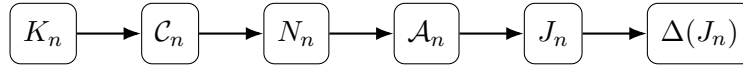
\begin{figure}[ht]
\centering
\begin{tikzpicture}[
  node distance=1.1cm and 0.85cm,
  box/.style={draw, rounded corners, align=center, inner sep=5pt, minimum height=8mm},
  arrow/.style={-Latex, thick}
]
\node[box] (K) {\(K_n\)};
\node[box, right=of K] (C) {\(\mathcal C_n\)};
\node[box, right=of C] (N) {\(N_n\)};
\node[box, right=of N] (A) {\(\mathcal A_n\)};
\node[box, right=of A] (J) {\(J_n\)};
\node[box, right=of J] (D) {\(\Delta(J_n)\)};
\draw[arrow] (K) -- (C);
\draw[arrow] (C) -- (N);
\draw[arrow] (N) -- (A);
\draw[arrow] (A) -- (J);
\draw[arrow] (J) -- (D);
\node[below=0.75cm of N, align=center] (eq) {\(K_n\simeq N_n\),\quad \(K_n\simeq \Delta(J_n)\),\quad \(\dim\Delta(J_n)\le 2\)};
\end{tikzpicture}
\caption{The overlap-control mechanism. The canonical cover by full star/top simplices gives \(K_n\simeq N_n\). The anchor-cover data then define an intersection poset \(J_n\), whose order complex is homotopy equivalent to \(K_n\) and has dimension at most \(2\).}
\label{fig:overlap-pipeline}
\end{figure}

Before entering the details, it is useful to emphasize the direction of the reduction. The complex \(K_n\) is not replaced by a smaller subcomplex. Instead, it is replaced by successive combinatorial models that remember only the data relevant for homotopy. First, the canonical cover replaces individual cliques by containers and records how containers meet. Then the anchor construction records which original partitions force which container intersections. Finally, the closure operator removes redundant anchor descriptions and produces the poset \(J_n\). The final model \(\Delta(J_n)\) is therefore not a geometric subcomplex of \(K_n\); it is a homotopy-equivalent overlap model.

\begin{table}[ht]
\centering
\small
\begin{adjustbox}{max width=\textwidth}
\begin{tabular}{llll}
\toprule
Stage & Input & Output & What is kept \\
\midrule
canonical cover & simplices of \(K_n\) & containers \(\mathcal C_n\) & star/top organization \\
nerve & cover \(\mathcal C_n\) & \(N_n\) & which containers intersect \\
anchors & partitions \(\lambda\) & simplices \(A_\lambda\subset N_n\) & containers passing through \(\lambda\) \\
anchor intersections & sets \(S\subset\Par(n)\) & \(A_S\) & containers containing all of \(S\) \\
closure & simplices \(S\in P_n\) & \(\operatorname{cl}_{\mathcal A}(S)\) & vertices forced by \(A_S\) \\
intersection poset & anchor intersections & \(J_n\) & nonredundant overlap data \\
order complex & poset \(J_n\) & \(\Delta(J_n)\) & chains of overlaps \\
\bottomrule
\end{tabular}
\end{adjustbox}
\caption{The internal steps of the overlap-control mechanism. The key point is that the homotopy type is preserved while the data are reorganized from raw simplices to controlled overlap chains.}
\label{tab:overlap-stages}
\end{table}

\subsection{The canonical cover and the first nerve}

Let \(\mathcal C_n\) be the collection of all distinct full star-simplices and full top-simplices in \(K_n\). By the star/top containment theorem,
\[
K_n=\bigcup_{U\in\mathcal C_n}U.
\]
Thus \(\mathcal C_n\) is a canonical cover of \(K_n\) by simplices. Since \(\Par(n)\) is finite, the complex \(K_n\), the cover \(\mathcal C_n\), and all nerves and posets constructed below are finite. Since every \(U\in\mathcal C_n\) is a simplex, it is contractible. Moreover, if \(U_0,\ldots,U_r\in\mathcal C_n\), then \(U_0\cap\cdots\cap U_r\) is a set of vertices contained in each simplex \(U_i\). When nonempty, such a set is a face of each \(U_i\), hence a simplex of \(K_n\), and therefore contractible. Thus every nonempty finite intersection of members of \(\mathcal C_n\) is contractible.

Let
\[
N_n=N(\mathcal C_n)
\]
be the nerve of this cover. Its vertices are the containers \(U\in\mathcal C_n\), and a finite set \(\{U_0,\ldots,U_r\}\) spans a simplex of \(N_n\) precisely when
\[
U_0\cap\cdots\cap U_r\ne\varnothing.
\]
Equivalently, a simplex of \(N_n\) records a finite family of canonical containers having at least one common partition.

\begin{backgroundtheorem}[Good-cover reduction]\label{thm:good-cover}
For every \(n\),
\[
K_n\simeq N_n.
\]
\end{backgroundtheorem}

\begin{proof}[Proof sketch]
The family \(\mathcal C_n\) is a good cover in the simplicial sense: its members are contractible and every nonempty finite intersection is contractible. The nerve lemma therefore gives a homotopy equivalence \(K_n\simeq N(\mathcal C_n)=N_n\) \cite{BorsukNerve,BjornerTopologicalMethods,KozlovCombinatorialAlgebraicTopology}.
\end{proof}

\begin{remark}
The passage from \(K_n\) to \(N_n\) removes the internal simplicial structure of each full star/top container and remembers only which containers meet. This is already a major conceptual shift: local simplex size is replaced by container-overlap data. However, \(N_n\) itself may still be complicated. The next steps explain why its homotopy type is nevertheless controlled by a shallow intersection poset.
\end{remark}

\subsection{Anchor simplices}

For each partition \(\lambda\vdash n\), define
\[
A_\lambda=\{U\in\mathcal C_n:\lambda\in U\}.
\]
Thus \(A_\lambda\) is the set of all full star/top containers passing through \(\lambda\). Since all these containers have \(\lambda\) in common, they span a simplex of the nerve \(N_n\). We use the same notation \(A_\lambda\) for this simplex of \(N_n\).

The indexed family
\[
\mathcal A_n=(A_\lambda)_{\lambda\vdash n}
\]
is the anchor cover of \(N_n\). It is indexed by the original vertices of \(K_n\), rather than by the containers. When speaking only about the union of the cover we may suppress the indexing, but the indexing is kept when forming its nerve.

\begin{proposition}[The anchors cover the nerve]\label{prop:anchors-cover}
The indexed family \(\mathcal A_n\) covers \(N_n\):
\[
N_n=\bigcup_{\lambda\vdash n}A_\lambda.
\]
\end{proposition}

\begin{proof}
Let \(\sigma=\{U_0,\ldots,U_r\}\) be a simplex of \(N_n\). By definition of the nerve,
\[
U_0\cap\cdots\cap U_r\ne\varnothing.
\]
Choose \(\lambda\) in this common intersection. Then every \(U_i\) contains \(\lambda\), hence \(U_i\in A_\lambda\) for all \(i\). Therefore \(\sigma\subseteq A_\lambda\).
\end{proof}

For a nonempty finite set \(S\subseteq\Par(n)\), define
\[
A_S=\bigcap_{\lambda\in S}A_\lambda
=
\{U\in\mathcal C_n:S\subseteq U\}.
\]
Thus \(A_S\) is the simplex of \(N_n\), if nonempty, whose vertices are precisely the full star/top containers containing all partitions in \(S\).

The following elementary criterion is useful throughout the reduction.

\begin{backgroundproposition}[Anchor-intersection criterion]\label{prop:anchor-criterion}
Let \(S\subseteq\Par(n)\) be finite and nonempty. Then
\[
A_S\ne\varnothing
\]
if and only if \(S\) is a simplex of \(K_n\). Whenever \(A_S\ne\varnothing\), the intersection \(A_S\) is itself a simplex of \(N_n\).
\end{backgroundproposition}

\begin{proof}
If \(A_S\ne\varnothing\), then some container \(U\in\mathcal C_n\) contains every vertex of \(S\). Since \(U\) is a simplex of \(K_n\), the set \(S\) is a simplex of \(K_n\).

Conversely, if \(S\) is a simplex of \(K_n\), the star/top containment theorem gives a container \(U\in\mathcal C_n\) with \(S\subseteq U\). Hence \(U\in A_S\), so \(A_S\ne\varnothing\).

Finally, assume \(A_S\ne\varnothing\). Its vertices are containers containing \(S\). Any finite subfamily of these containers has common intersection containing \(S\), hence spans a simplex of \(N_n\). Therefore all vertices of \(A_S\) span a simplex of \(N_n\).
\end{proof}

\begin{corollary}[The anchor cover is good]\label{cor:anchor-good}
The indexed family \(\mathcal A_n\) is a good cover of \(N_n\). Moreover, the nerve of this indexed cover is canonically isomorphic to \(K_n\).
\end{corollary}

\begin{proof}
Each \(A_\lambda\) is a simplex of \(N_n\), and every nonempty finite intersection \(A_S\) is a simplex by Proposition~\ref{prop:anchor-criterion}. Hence the cover is good. Its nerve has one vertex for each partition \(\lambda\vdash n\), and a finite set \(S\) spans a simplex precisely when \(A_S\ne\varnothing\). By Proposition~\ref{prop:anchor-criterion}, this is precisely the condition that \(S\) be a simplex of \(K_n\).
\end{proof}

\begin{remark}[Why a second nerve is not enough]
The preceding corollary is deliberately included because it prevents a common misunderstanding. If one simply takes the nerve of the anchor cover, one recovers \(K_n\). The low-dimensional model does not come from replacing \(N_n\) by this second nerve. It comes from retaining the actual intersection data of the anchor cover and organizing those intersections into a poset. The reduction is therefore an intersection-poset reduction, not merely a second application of the nerve construction.
\end{remark}

\subsection{Rigidity of anchor intersections}

The decisive input is that nonempty intersections of anchors are much more rigid than the anchors themselves. A single anchor \(A_\lambda\) may be large, because many full star/top containers may pass through \(\lambda\). But if two or more partitions are prescribed, the number of containers containing all of them becomes sharply constrained.

The geometric reason is the star/top classification of cliques. Relative to a vertex \(\lambda\), a full star container fixes one removable corner and allows several addable targets; a full top container fixes one addable corner and allows several removable sources. Therefore, if a container contains \(\lambda\) and two different transfers with the same removable corner, it is forced to be the corresponding full star container. Dually, if it contains \(\lambda\) and two different transfers with the same addable corner, it is forced to be the corresponding full top container. This forcing is the local source of overlap rigidity.

For the present synthesis we use the following consequence, proved as part of the star/top classification and anchor analysis in the topological paper \cite[Theorem~5.7]{LyudogovskiyHomotopyPartitionGraph}. The proof idea below recalls the mechanism, but the proposition is used here as an imported background result.

\begin{backgroundproposition}[Anchor-intersection rigidity]\label{prop:anchor-rigidity}
Let \(S\subseteq\Par(n)\) be nonempty and suppose \(A_S\ne\varnothing\). Then:
\begin{enumerate}[label=\textup{(\arabic*)}]
\item if \(|S|=1\), the simplex \(A_S=A_\lambda\) may have arbitrary dimension;
\item if \(|S|=2\), then \(A_S\) has at most three vertices;
\item if \(|S|\ge 3\), then \(A_S\) is a single vertex of \(N_n\).
\end{enumerate}
\end{backgroundproposition}

\begin{figure}[ht]
\centering
\[
A_S=\bigcap_{\lambda\in S}A_\lambda
\qquad
\begin{array}{c|c}
|S| & \text{possible size of }A_S\\
\hline
1 & \text{may be large}\\
2 & \text{at most three vertices}\\
\ge 3 & \text{one vertex}
\end{array}
\]
\caption{Anchor-intersection rigidity. Single anchors may be large, but intersections of two anchors are small and intersections of three or more anchors collapse to a point. This rigidity is the input for the no-long-chain theorem for \(J_n\).}
\label{fig:anchor-rigidity}
\end{figure}

\begin{proof}[Proof idea, recalled from \cite{LyudogovskiyHomotopyPartitionGraph}]
The first case is only the definition of an anchor. For the second case, suppose \(S=\{\lambda,\mu\}\) with \(\mu\sim\lambda\), and write \(\mu=\lambda(c_\mu\to a_\mu)\). The classification of containers through \(\lambda\) says that any vertex \(U\in A_\lambda\) is one of three kinds: a full star container \(\Sigma^{\mathrm{star}}_{\max}(\lambda,c)\), a full top container \(\Sigma^{\mathrm{top}}_{\max}(\lambda,a)\), or a low-dimensional edge container \(\{\lambda,\nu\}\) when that edge itself occurs as a full star/top container. If such a \(U\) also contains \(\mu\), then necessarily
\[
c=c_\mu,
\qquad
 a=a_\mu,
\qquad
\text{or}\qquad
 \nu=\mu,
\]
respectively. Hence \(A_\lambda\cap A_\mu\) has at most the three vertices
\[
\Sigma^{\mathrm{star}}_{\max}(\lambda,c_\mu),
\qquad
\Sigma^{\mathrm{top}}_{\max}(\lambda,a_\mu),
\qquad
\{\lambda,\mu\},
\]
where some may be absent or may coincide as sets.

For \(|S|\ge 3\), the star/top clique classification implies that \(S\) is star-type or top-type relative to a suitable vertex of \(S\). In the star case it contains two distinct transfers with a common removable corner. By the forcing lemmas for full star/top containers, any full star/top container containing all of \(S\) must then be the unique full star container determined by that common removable corner. In the top case the dual argument forces the unique full top container determined by the common addable corner. Hence \(A_S\) has exactly one vertex.
\end{proof}

This proposition is the visible local rigidity statement that neutralizes high-dimensional complexity at the overlap level. Large anchors are possible, but higher intersections are shallow. The next subsection converts this fact into a low-dimensional homotopy model.

\subsection{The intersection poset and the closure operator}

Let
\[
P_n=\{\text{nonempty simplices of }K_n\},
\]
ordered by inclusion. For \(S\in P_n\), the intersection \(A_S\) is nonempty by Proposition~\ref{prop:anchor-criterion}. Define
\[
\operatorname{cl}_{\mathcal A}(S)
=
\{\mu\in\Par(n): A_S\subseteq A_\mu\}.
\]
In words, \(\operatorname{cl}_{\mathcal A}(S)\) is the set of all partitions whose anchors contain the whole intersection \(A_S\). Equivalently, it is the largest set of original vertices that is forced by the anchor intersection \(A_S\).

\begin{proposition}[Anchor closure]\label{prop:anchor-closure}
The map
\[
\operatorname{cl}_{\mathcal A}:P_n\to P_n
\]
is a closure operator: it is extensive, monotone, and idempotent.
\end{proposition}

\begin{proof}
First, \(\operatorname{cl}_{\mathcal A}(S)\) is a simplex of \(K_n\). Indeed, if \(T\) is a finite subset of \(\operatorname{cl}_{\mathcal A}(S)\), then \(A_S\subseteq A_T\). Since \(A_S\ne\varnothing\), also \(A_T\ne\varnothing\), and Proposition~\ref{prop:anchor-criterion} implies that \(T\) is a simplex of \(K_n\).

The map is extensive because if \(\lambda\in S\), then \(A_S\subseteq A_\lambda\), so \(\lambda\in\operatorname{cl}_{\mathcal A}(S)\). It is monotone because \(S\subseteq T\) implies \(A_T\subseteq A_S\); hence every anchor containing \(A_S\) also contains \(A_T\). Finally, set \(C=\operatorname{cl}_{\mathcal A}(S)\). By extensivity, \(A_C\subseteq A_S\). Conversely, every \(\mu\in C\) satisfies \(A_S\subseteq A_\mu\), hence \(A_S\subseteq A_C\). Thus \(A_C=A_S\), and therefore
\[
\operatorname{cl}_{\mathcal A}(C)=\operatorname{cl}_{\mathcal A}(S)=C.
\]
So the map is idempotent.
\end{proof}

Now define the intersection poset
\[
J_n=\{A_S:S\in P_n\},
\]
where equal intersections are identified, ordered by inclusion.

The closure operator explains why \(J_n\) is not merely an auxiliary bookkeeping device. It is the fixed-part model of the anchor-intersection structure.

\begin{proposition}[Fixed part and intersection poset]\label{prop:fixed-intersection-poset}
The assignment
\[
S\longmapsto A_S
\]
induces an order-reversing bijection
\[
\operatorname{Fix}(\operatorname{cl}_{\mathcal A})^{\mathrm{op}}
\cong
J_n,
\]
where \(\operatorname{Fix}(\operatorname{cl}_{\mathcal A})\) is the set of fixed simplices of the closure operator.
\end{proposition}

\begin{proof}
For \(X\in J_n\), define
\[
\Psi(X)=\{\mu\in\Par(n):X\subseteq A_\mu\}.
\]
If \(X=A_S\), then \(\Psi(X)=\operatorname{cl}_{\mathcal A}(S)\). Thus fixed simplices are exactly recovered from their intersections. Conversely, if \(T=\operatorname{cl}_{\mathcal A}(S)\), the proof of Proposition~\ref{prop:anchor-closure} gives \(A_T=A_S\). Hence the two assignments are inverse after restricting to fixed simplices. Since inclusion of simplices reverses inclusion of the corresponding anchor intersections, the bijection is order-reversing.
\end{proof}

A standard theorem on closure operators on finite posets says that the order complex of a finite poset deformation retracts onto the order complex of the fixed subposet of a closure operator \cite{BjornerTopologicalMethods,KozlovCombinatorialAlgebraicTopology,BarmakFiniteSpaces}. Applying this to \(P_n\) gives the following model.

\begin{backgroundtheorem}[Anchor-poset reduction]\label{thm:anchor-poset}
For every \(n\),
\[
K_n\simeq \Delta(J_n).
\]
Consequently, since \(K_n\simeq N_n\), also
\[
N_n\simeq \Delta(J_n).
\]
\end{backgroundtheorem}

\begin{proof}[Proof sketch]
The order complex \(\Delta(P_n)\) is the barycentric subdivision of \(K_n\), so \(\Delta(P_n)\simeq K_n\). By the closure-operator theorem,
\[
\Delta(P_n)\simeq \Delta(\operatorname{Fix}(\operatorname{cl}_{\mathcal A})).
\]
By Proposition~\ref{prop:fixed-intersection-poset}, the fixed subposet is anti-isomorphic to \(J_n\), and a poset and its opposite have canonically isomorphic order complexes. Hence
\[
K_n\simeq \Delta(J_n).
\]
The equivalence with \(N_n\) follows from Background Theorem~\ref{thm:good-cover}.
\end{proof}

\begin{remark}
This is the formal point at which high-dimensional local simplices stop being the relevant global object. The complex \(K_n\) is replaced, up to homotopy, by the order complex of a poset whose elements are intersections of anchors. The dimension of this model is determined by the length of strict inclusion chains among such intersections.
\end{remark}

\subsection{Chain length in the intersection poset}

We now explain why \(\Delta(J_n)\) has dimension at most \(2\). A simplex of \(\Delta(J_n)\) is a strict chain
\[
X_0\subsetneq X_1\subsetneq\cdots\subsetneq X_r
\]
in \(J_n\). Therefore bounding the dimension of \(\Delta(J_n)\) is the same as bounding the length of strict chains of anchor intersections.

The following elementary consequences of Proposition~\ref{prop:anchor-rigidity} make the chain bound transparent.

\begin{lemma}\label{lem:non-singleton-intersections}
Every non-singleton element of \(J_n\) is either an anchor \(A_\lambda\) or a pairwise overlap
\[
A_{\lambda,\mu}=A_\lambda\cap A_\mu
\]
for some edge \(\lambda\mu\in E(G_n)\).
\end{lemma}

\begin{proof}
Write \(X=A_S\). If \(|S|\ge 3\), then Proposition~\ref{prop:anchor-rigidity} says that \(A_S\) is a singleton. Hence a non-singleton \(X\) must come from \(|S|=1\) or \(|S|=2\). These give respectively \(A_\lambda\) and \(A_{\lambda,\mu}\). In the second case \(A_{\lambda,\mu}\ne\varnothing\), so \(\{\lambda,\mu\}\) is a simplex of \(K_n\), equivalently \(\lambda\sim\mu\) in \(G_n\).
\end{proof}

\begin{lemma}\label{lem:below-anchor}
Let \(X\in J_n\) be non-singleton and suppose \(X\subseteq A_\lambda\). Then either
\[
X=A_\lambda,
\]
or
\[
X=A_{\lambda,\mu}
\]
for some neighbor \(\mu\sim\lambda\).
\end{lemma}

\begin{proof}
Write \(X=A_S\). Since \(X\subseteq A_\lambda\), we have
\[
X=A_S=A_S\cap A_\lambda=A_{S\cup\{\lambda\}}.
\]
Because \(X\) is non-singleton, Proposition~\ref{prop:anchor-rigidity} implies \(|S\cup\{\lambda\}|\le 2\). If \(S\cup\{\lambda\}=\{\lambda\}\), then \(X=A_\lambda\). If \(S\cup\{\lambda\}=\{\lambda,\mu\}\), then
\[
X=A_{\lambda,\mu}.
\]
This includes the apparently separate case \(X=A_\mu\) with \(A_\mu\subseteq A_\lambda\): in that situation
\[
A_\mu=A_\mu\cap A_\lambda=A_{\lambda,\mu},
\]
so it is already a pairwise-overlap case. Since \(A_{\lambda,\mu}\ne\varnothing\), the anchor criterion implies that \(\lambda\mu\) is an edge of \(G_n\) when \(\mu\ne\lambda\).
\end{proof}

\begin{lemma}\label{lem:below-pair}
Let \(X\in J_n\) be non-singleton and suppose
\[
X\subseteq A_{\lambda,\mu}=A_\lambda\cap A_\mu.
\]
Then
\[
X=A_{\lambda,\mu}.
\]
\end{lemma}

\begin{proof}
Write \(X=A_S\). Since \(X\subseteq A_\lambda\cap A_\mu\),
\[
X=A_S=A_S\cap A_\lambda\cap A_\mu=A_{S\cup\{\lambda,\mu\}}.
\]
Again \(X\) is non-singleton, so Proposition~\ref{prop:anchor-rigidity} implies \(|S\cup\{\lambda,\mu\}|\le 2\). Since this set already contains the two distinct partitions \(\lambda\) and \(\mu\), it must equal \(\{\lambda,\mu\}\), and therefore
\[
X=A_{\lambda,\mu}.
\]
Thus possible nested anchors do not create a further case: if, for example, \(X=A_\nu\subseteq A_{\lambda,\mu}\), then the displayed equality gives \(A_\nu=A_{\lambda,\mu,\nu}\); non-singleton rigidity forces \(\nu\in\{\lambda,\mu\}\), and hence \(X=A_{\lambda,\mu}\).
\end{proof}

\begin{backgroundtheorem}[Low-dimensional overlap model]\label{thm:low-dimensional-model}
For every \(n\), every strict chain in \(J_n\) has at most three elements. Equivalently,
\[
\dim\Delta(J_n)\le 2.
\]
Consequently, \(K_n\) has a homotopy model of dimension at most \(2\):
\[
K_n\simeq\Delta(J_n),
\qquad
\dim\Delta(J_n)\le 2.
\]
\end{backgroundtheorem}

\begin{proof}
Let
\[
X_0\subsetneq X_1\subsetneq\cdots\subsetneq X_r
\]
be a strict chain in \(J_n\). If \(X_r\) is a singleton, then \(r=0\). Suppose \(X_r\) is not a singleton. By Lemma~\ref{lem:non-singleton-intersections}, \(X_r\) is either an anchor \(A_\lambda\) or a pairwise overlap \(A_{\lambda,\mu}\).

If \(X_r=A_{\lambda,\mu}\), then Lemma~\ref{lem:below-pair} says that no proper non-singleton element lies below it. Therefore any proper element below it is a singleton, and the chain has length at most \(1\).

If \(X_r=A_\lambda\), then Lemma~\ref{lem:below-anchor} says that any proper non-singleton element below it is a pairwise overlap \(A_{\lambda,\mu}\). In particular, a nested anchor \(A_\mu\subsetneq A_\lambda\), if it occurs, is not a new level: it equals \(A_{\lambda,\mu}\). By Lemma~\ref{lem:below-pair}, nothing non-singleton lies properly below such a pairwise overlap. Hence the longest possible chain has the form
\[
\{U\}\subsetneq A_{\lambda,\mu}\subsetneq A_\lambda,
\]
which has length \(2\). Thus every strict chain in \(J_n\) has length at most \(2\), and \(\dim\Delta(J_n)\le 2\).

The homotopy equivalence \(K_n\simeq\Delta(J_n)\) is Background Theorem~\ref{thm:anchor-poset}.
\end{proof}

\begin{remark}[Nested anchors]
The proof allows the possibility that distinct anchors may be nested. Such a nesting does not create an additional level in \(J_n\). Indeed, if \(A_\mu\subseteq A_\lambda\), then \(A_\mu=A_{\lambda,\mu}\), so the smaller anchor is already counted as a pairwise overlap. Similarly, if \(A_\nu\subseteq A_{\lambda,\mu}\) and the intersection is non-singleton, then \(A_\nu=A_{\lambda,\mu}\). This is the point that prevents chains of anchors from extending beyond the three-level pattern
\[
\{U\}\subsetneq A_{\lambda,\mu}\subsetneq A_\lambda.
\]
\end{remark}

\begin{remark}[What is bounded]
The theorem does not say that anchors are small. A single anchor \(A_\lambda\) may be a simplex of high dimension, and full star/top containers may also have high dimension. What is bounded is the depth of strict chains of anchor intersections. This is why the global model has dimension at most \(2\) even though the original clique complex can contain high-dimensional simplices.
\end{remark}

\subsection{From overlap depth to global topology}

Background Theorem~\ref{thm:low-dimensional-model} is the technical heart of the synthesis. It explains why the global homotopy dimension remains low even when the original clique complex contains high-dimensional simplices. The important bounded quantity is not the dimension of a full star/top container, nor the dimension of an anchor simplex. Those may grow. The bounded quantity is the depth of strict overlap chains after passing to the anchor-intersection poset.

The global topology theorem then adds two further established inputs: \(K_n\) is simply connected, and its reduced homology is concentrated in degree \(2\). Combined with the \(2\)-dimensional model, these facts yield
\[
K_n\simeq \bigvee^{b_n}S^2,
\qquad
b_n=\chi(K_n)-1.
\]
The overlap-control mechanism can be summarized by the following dependency chain:
\[
\begin{aligned}
\mathcal C_n\text{ is a good cover}
&\Longrightarrow K_n\simeq N_n,\\
\text{anchor intersections}
&\Longrightarrow J_n,\\
\text{closure-operator reduction}
&\Longrightarrow K_n\simeq \Delta(J_n),\\
\text{chain length in }J_n\le 2
&\Longrightarrow \dim\Delta(J_n)\le 2.
\end{aligned}
\]
In this precise sense, the global model measures overlap depth rather than local simplex size. This is the structural reason why local high-dimensionality does not create high-dimensional global homology.

\section{Why Local Complexity and Global Simplicity Are Compatible}\label{sec:compatible}

We now assemble the two sides of the argument. Sections~\ref{sec:basic} and~\ref{sec:local} described the local and morphological side. Sections~\ref{sec:containers} and~\ref{sec:overlap} described the global topological side. The core answer is:
\[
\begin{array}{c}
\boxed{\text{local complexity grows inside canonical containers,}}\\[4pt]
\boxed{\text{global topology is controlled by container overlaps.}}
\end{array}
\]

\begin{figure}[ht]
\centering
\small
\begin{tabular}{c|c}
\textbf{Unbounded local size} & \textbf{Bounded global overlap depth}\\
\hline
\(\dim K_n\) is unbounded & \(\dim\Delta(J_n)\le 2\)\\
large local simplices & short chains in \(J_n\)\\
thick star/top containers & \(K_n\simeq\bigvee^{b_n}S^2\)\\
\end{tabular}
\caption{The separation between local simplex size and global overlap depth. The simplicial dimension of \(K_n\) is unbounded, but the order complex of the anchor-intersection poset is always at most two-dimensional.}
\label{fig:local-vs-global-depth}
\end{figure}

\subsection{Two different meanings of dimension}

The apparent tension comes partly from using the word ``dimension'' in different senses. The simplicial dimension of \(K_n\) is
\[
\dim K_n=\max\{\dim\sigma:\sigma\subseteq K_n\}.
\]
This is the dimension of the largest clique in \(G_n\). By contrast, the homotopical dimension relevant to the global theorem is the dimension of a homotopy model for \(K_n\). Background Theorem~\ref{thm:low-dimensional-model} gives a model \(K_n\simeq\Delta(J_n)\) with \(\dim\Delta(J_n)\le 2\).

\begin{corollary}[Simplicial dimension and homotopy dimension separate]\label{cor:dimension-separation}
The family \(K_n\) has unbounded simplicial dimension, but every \(K_n\) admits a homotopy model of dimension at most \(2\). In particular, high-dimensional simplices in \(K_n\) do not imply high-dimensional homotopy type.
\end{corollary}

\begin{proof}
The unboundedness of simplicial dimension follows from Proposition~\ref{prop:staircase}. The existence of a two-dimensional homotopy model follows from Background Theorem~\ref{thm:low-dimensional-model}.
\end{proof}

A large simplex in \(K_n\) records a large family of mutually adjacent vertices in \(G_n\). In the local transfer picture, such a family arises from a large compatible family of elementary transfers. Thus large simplices are local witnesses of thickness. They contribute to simplex layers and thick zones, but by themselves they do not carry reduced homology.

\begin{synthesisprinciple}[Local thickness is not global obstruction]\label{prin:thickness-not-obstruction}
Large local simplices should be interpreted as local thickness, not as automatic global homological obstruction. They may affect simplex counts, Euler characteristic, bouquet rank, and local or regional morphology, but they do not by themselves force nonzero homology in their own dimension.
\end{synthesisprinciple}

\subsection{The failure of the naive expectation}

A natural first expectation might be that as local simplex dimension grows, global topology should become more complicated. For the complexes \(K_n\), this expectation fails. More precisely, the implication
\[
\dim K_n\text{ large}
\Longrightarrow
\widetilde H_k(K_n)\ne 0\text{ for some large }k
\]
does not hold. Indeed, for every \(n\),
\[
\widetilde H_k(K_n)=0\qquad(k\ne 2).
\]

\begin{corollary}[No high-dimensional homology from local high-dimensionality]\label{cor:no-high-homology}
For every \(n\),
\[
\widetilde H_k(K_n)=0\qquad(k\ge 3).
\]
This remains true even for values of \(n\) for which \(K_n\) contains simplices of dimension \(k\) or larger.
\end{corollary}

\begin{proof}
The vanishing follows from the global bouquet theorem, or already from the existence of a two-dimensional homotopy model. The existence of high-dimensional simplices for suitable \(n\) follows from Proposition~\ref{prop:staircase}.
\end{proof}

\subsection{The compatibility theorem}

\begin{synthesistheorem}[Local complexity and global simplicity are compatible through overlap control]\label{thm:compatibility}
For the family of partition graphs \(G_n\) and clique complexes \(K_n=\Cl(G_n)\), the following hold.
\begin{enumerate}[label=\textup{(\arabic*)}]
\item The local transfer type of a partition determines its local clique geometry, including degree, local clique number, and local simplex dimension.
\item The local simplex dimension of \(K_n\) is unbounded along the family.
\item Every simplex of \(K_n\) is contained in a contractible full star/top container.
\item The canonical cover by these containers is a good cover:
\[
K_n\simeq N_n.
\]
\item The anchor-cover/intersection-poset reduction gives
\[
N_n\simeq\Delta(J_n),
\qquad
\dim\Delta(J_n)\le 2.
\]
\item The established global topology theorem gives
\[
K_n\simeq \bigvee^{b_n}S^2,
\qquad
b_n=\chi(K_n)-1.
\]
\end{enumerate}
Therefore the growth of local and morphological complexity in \(G_n\) is compatible with the stable global topological form of \(K_n\): local complexity grows inside contractible containers, while global topology is controlled by a low-dimensional overlap poset.
\end{synthesistheorem}

\begin{proof}
Statement (1) is the local transfer-type theorem recalled in Background Proposition~\ref{prop:local-transfer}. Statement (2) follows from the staircase construction in Proposition~\ref{prop:staircase}. Statement (3) follows from the star/top containment theorem. Statement (4) is the good-cover reduction. Statement (5) is the anchor-cover/intersection-poset reduction and Background Theorem~\ref{thm:low-dimensional-model}. Statement (6) is the bouquet theorem. The final conclusion is the synthesis of these results.
\end{proof}

Synthesis Theorem~\ref{thm:compatibility} packages established components into one local/global explanation; it is not an independent new classification theorem. 

\subsection{What has and has not been explained}

The synthesis explains why high-dimensional simplices occur, why they do not create high-dimensional homology, why the global topology is low-dimensional in character, and why the final homotopy type is a bouquet of two-spheres. It also clarifies what remains outside the present explanation. The mechanism does not yet give a closed conceptual formula for \(b_n\). It does not decompose \(\chi(K_n)\) into contributions from named morphological regions such as the axis, spine, rear region, central zone, support strata, or shell layers. Thus the qualitative local/global problem is explained by overlap control, while a numerical and morphological refinement problem remains.

\section{Numerical and Topological Consequences}\label{sec:numerical}

Once the overlap-control mechanism is in place, the qualitative topological form of \(K_n\) is fixed. The numerical consequences are developed in detail in the numerical topology paper \cite{LyudogovskiyNumericalTopology}; here they are used to locate the residual global problem:
\[
K_n\simeq \bigvee^{b_n}S^2.
\]
Thus the remaining global topological variation is concentrated in one integer, \(b_n\). Since \(b_n=\chi(K_n)-1\), the global topological problem becomes numerical.

\subsection{From homotopy type to one integer}

The Euler characteristic of a bouquet of \(b_n\) two-spheres is \(1+b_n\). Therefore
\[
b_n=\chi(K_n)-1.
\]

\begin{corollary}[Numerical collapse of the qualitative topological problem]\label{cor:numerical-collapse}
For every \(n\), the qualitative homotopy type of \(K_n\) is completely determined by the Euler characteristic:
\[
K_n\simeq \bigvee^{\chi(K_n)-1}S^2.
\]
Consequently, within this family, the global topological classification problem reduces to the numerical problem of computing \(\chi(K_n)\), or equivalently \(b_n\).
\end{corollary}

\subsection{Three numerical packages}

Let \(c_r(n)\) denote the number of \(r\)-cliques in \(G_n\), meaning complete subgraphs on \(r\) vertices. Such cliques are exactly the \((r-1)\)-simplices of \(K_n\). Then
\[
\chi(K_n)=\sum_{r\ge 1}(-1)^{r-1}c_r(n),
\]
and
\[
b_n=\sum_{r\ge 1}(-1)^{r-1}c_r(n)-1.
\]

Let \(d_r(n)\) denote the number of \(r\)-simplices of the nerve \(N_n\). Since \(K_n\simeq N_n\),
\[
b_n=\sum_{r\ge 0}(-1)^r d_r(n)-1.
\]
Finally, since \(K_n\simeq\Delta(J_n)\) and \(\dim\Delta(J_n)\le 2\),
\[
b_n=f_0(\Delta(J_n))-f_1(\Delta(J_n))+f_2(\Delta(J_n))-1.
\]

\begin{table}[ht]
\centering
\begin{tabular}{>{\raggedright\arraybackslash}p{0.22\textwidth} >{\centering\arraybackslash}p{0.38\textwidth} >{\raggedright\arraybackslash}p{0.28\textwidth}}
\toprule
Package & Formula & Interpretation \\
\midrule
Clique package & \(\displaystyle b_n=\sum_{r\ge1}(-1)^{r-1}c_r(n)-1\) & raw simplicial enumeration in \(G_n\) \\[2mm]
Nerve package & \(\displaystyle b_n=\sum_{r\ge0}(-1)^r d_r(n)-1\) & container-overlap enumeration \\[2mm]
Poset package & \(\displaystyle b_n=f_0-f_1+f_2-1\) & low-dimensional chain enumeration in \(\Delta(J_n)\) \\
\bottomrule
\end{tabular}
\caption{Three equivalent numerical descriptions of the bouquet rank \(b_n\).}
\label{tab:numerical-packages}
\end{table}

These packages compute the same integer, but they organize the enumeration differently. The clique package sees the raw simplicial structure of \(K_n\). The nerve package sees the overlap structure of canonical containers. The poset package sees the low-dimensional overlap model.

For this reason the numerical topology of \(K_n\) should not be viewed as a mere afterthought to the bouquet theorem. The bouquet theorem says that the qualitative homotopy type is fixed once \(b_n\) is known. It does not explain why the alternating clique count takes the values it does, how those values are distributed across support regimes, or which overlap configurations contribute positively or negatively to the Euler characteristic. These are genuinely combinatorial questions, and they remain meaningful precisely because the qualitative topological classification has already been simplified.

\begin{synthesisprinciple}[Morphology survives globally through enumeration]\label{prin:morphology-enumeration}
The local and regional morphology of \(G_n\) can influence the numerical global topology of \(K_n\) through enumerative invariants such as clique counts, nerve counts, intersection-poset chains, Euler characteristic, and bouquet rank. But the qualitative form
\[
K_n\simeq \bigvee^{b_n}S^2
\]
is fixed by the overlap-control mechanism.
\end{synthesisprinciple}

\subsection{Why the three packages are not redundant}\label{subsec:packages-not-redundant}

Although the three formulas for \(b_n\) compute the same integer, they are not redundant from the point of view of explanation. The clique package is closest to the raw graph-theoretic object: it counts complete subgraphs of \(G_n\). The nerve package is closer to the topology proof: it counts finite subfamilies of canonical containers with nonempty common intersection. The poset package is closest to the low-dimensional model: it counts chains in \(J_n\), and hence uses only vertices, edges, and triangles of \(\Delta(J_n)\).

Each package discards different information. Passing from clique counts to nerve counts forgets the internal geometry of individual containers. Passing from the nerve to the intersection poset removes further redundant overlap data, retaining the closure-stable overlap structure. The equality of the resulting Euler characteristics is topological, but the differences among the packages remain combinatorially meaningful.

This is why the numerical topology of \(K_n\) is not exhausted by the identity \(b_n=\chi(K_n)-1\). The identity says which number matters. It does not say which enumeration gives the most conceptual explanation of that number.

The reduction to \(\chi(K_n)\) should not be confused with solving the enumeration problem. The homotopy type is simple in form, but its rank may encode complicated partition-theoretic and morphological information.

\section{Open Conceptual Problems and Outlook}\label{sec:open}

This final conceptual section separates what the present synthesis explains from what it leaves open. The qualitative compatibility problem is explained by the overlap-control mechanism. What remains concerns numerical interpretation, morphological decomposition, and possible extensions.

\subsection{Morphological formula for Euler characteristic}

The Euler characteristic can be computed by clique counts, nerve counts, or the low-dimensional poset model. These formulas are exact, but still primarily enumerative.

\begin{openproblem}[Morphological decomposition of the Euler characteristic]\label{op:chi-morphological}
Find a structural decomposition of \(\chi(K_n)\) in terms of regions or strata of \(G_n\), such as the central region, rear region, self-conjugate axis, spine, support strata, simplex layers, shells, and boundary layers.
\end{openproblem}

The difficulty is that Euler characteristic is not generally additive over overlapping morphological regions without inclusion-exclusion. Such a formula would need to account not only for regional simplex counts, but also for the overlaps among regions.

\subsection{The conceptual meaning of the bouquet rank}

The integer \(b_n\) has a clear topological meaning: it is the number of two-spherical summands in the bouquet model. It also has exact enumerative descriptions. But these formulas do not yet give a direct conceptual interpretation of what \(b_n\) counts in partition-theoretic or morphological terms.

\begin{openproblem}[Partition-theoretic meaning of the bouquet rank]\label{op:bn-meaning}
Give a direct combinatorial or morphological interpretation of \(b_n\). For example, determine whether \(b_n\) can be interpreted in terms of overlap defects, minimal two-dimensional cycles, shell interactions, container-gluing obstructions, phase-boundary loops, or partition-theoretic configurations.
\end{openproblem}

\subsection{Phase boundaries and the nerve}

The morphological papers identify degree transitions \cite{LyudogovskiyDegreeLandscape,LyudogovskiyDegreeTheory}, simplex-layer transitions \cite{LyudogovskiySimplexStratification,LyudogovskiySimplexLayers}, support jumps \cite{LyudogovskiySupportJumps}, shell boundaries \cite{LyudogovskiySimplicialShells}, rear-central transitions \cite{LyudogovskiyBoundaryRear}, and directional-regime changes \cite{LyudogovskiyDirectionalGeometry}. These boundaries live naturally in \(G_n\). The global topology, however, is controlled by \(N_n\) and \(J_n\).

\begin{openproblem}[Morphological phase boundaries in \(N_n\) and \(J_n\)]\label{op:phase-boundaries-nerve}
Determine whether major phase boundaries in \(G_n\) correspond to identifiable features of the canonical nerve \(N_n\) or the intersection poset \(J_n\). Possible signatures include changes in container density, clusters of intersections, chains in \(J_n\), two-dimensional cycles in \(\Delta(J_n)\), or localized contributions to \(\chi(K_n)\).
\end{openproblem}

\subsection{The self-conjugate axis and global topology}

The self-conjugate axis is one of the main organizing structures in the morphology of \(G_n\) \cite{LyudogovskiyAxialMorphology,LyudogovskiyGrowingDiscreteGeometricObject}. It is visible in axial morphology, central concentration, spine behavior, and conjugation symmetry. However, its direct role in the global topology of \(K_n\) is not yet fully clear.

\begin{openproblem}[Axis signatures in global topology]\label{op:axis-signature}
Determine whether the self-conjugate axis has a direct signature in \(\chi(K_n)\), \(b_n\), \(N_n\), \(J_n\), or \(\Delta(J_n)\). More specifically, determine whether conjugation symmetry induces useful decompositions or symmetries of the canonical cover, the nerve, or the intersection poset.
\end{openproblem}

\subsection{Morphogenesis of overlap objects}

The present synthesis suggests that one should study the morphogenesis of the overlap-control objects themselves:
\[
\mathcal C_n,
\qquad
N_n,
\qquad
J_n,
\qquad
\Delta(J_n).
\]

\begin{openproblem}[Morphogenesis of the canonical cover and intersection poset]\label{op:overlap-morphogenesis}
Describe how \(\mathcal C_n\), \(N_n\), and \(J_n\) evolve as \(n\) increases. Identify which changes in \(G_n\) are responsible for changes in \(\chi(K_n)\), \(b_n\), \(f(N_n)\), and \(f(\Delta(J_n))\).
\end{openproblem}

\subsection{Universality of overlap control}

The family \(G_n\) is special. Its topology depends on the elementary transfer rule, the partition structure, and the star/top clique classification. Nevertheless, the mechanism suggests a broader question.

\begin{openproblem}[Universality of the local/global separation mechanism]\label{op:universality}
Identify other graph families for which the following pattern holds:
\[
\text{large local cliques}
\Longrightarrow
\text{contractible canonical containers}
\Longrightarrow
\text{low-dimensional overlap model}.
\]
Possible test families include transfer graphs built from integer partitions, compositions, Young diagrams, bounded partitions, colored partitions, and redistribution systems with constraints.
\end{openproblem}

\subsection{Possible future refinements}\label{subsec:future-refinements}

Several refinements would make the bridge between morphology and topology more direct. They are not needed for the qualitative explanation proved in the present paper, but they suggest concrete directions for subsequent work.

\begin{enumerate}[label=\textup{(\arabic*)}, leftmargin=*]
\item \emph{Regional Euler packages.} Given a morphological decomposition of \(G_n\), one may try to associate subcomplexes or subfamilies of containers to the regions and express \(\chi(K_n)\) by an inclusion--exclusion formula over these pieces. The difficulty is that the morphologically meaningful regions are not usually disjoint, and their overlaps may be topologically significant.

\item \emph{Filtered nerves.} One may filter the canonical cover \(\mathcal C_n\) by support size, local simplex dimension, distance from the self-conjugate axis, shell index, or another morphological coordinate. This would produce filtered nerves whose graded Euler data could indicate how different regions contribute to the final bouquet rank.

\item \emph{Weighted intersection posets.} The poset \(J_n\) can be decorated by data inherited from its elements: container dimensions, star/top type, support range, axis distance, degree levels, or local thickness. Such weighted posets would retain the low-dimensional overlap model while remembering more morphology.

\item \emph{Overlap-defect invariants.} Since \(b_n\) is produced by the way containers fail to glue in a contractible pattern, it is natural to seek invariants measuring non-tree-like or non-collapsible features of the overlap structure. Such invariants may give a more conceptual interpretation of the bouquet rank.

\item \emph{Conjugation-equivariant models.} Partition conjugation acts on \(G_n\) and on the morphology of the graph. It is natural to ask whether the canonical cover, the nerve, or the intersection poset admits useful equivariant decompositions, and whether the self-conjugate axis has a visible signature in such decompositions.
\end{enumerate}

\section{Conclusion}\label{sec:conclusion}

The purpose of this paper has been to synthesize the local, morphological, and topological results obtained for the partition graphs \(G_n\) and their clique complexes \(K_n=\Cl(G_n)\) across the preceding papers of the cycle. The central question was why \(G_n\) develops increasing local and morphological complexity while the global homotopy type of \(K_n\) remains simple.

The answer is a clean separation of mechanisms. On the local side, the elementary transfer rule produces ordered local transfer types
\[
\mathfrak t(\lambda)=(t;\alpha;\beta),
\]
which determine the local transfer graph \(B(\lambda)\), the induced neighborhood graph
\[
G_n[N_{G_n}(\lambda)]\cong L(B(\lambda)),
\]
and local invariants such as degree, local clique number, and local simplex dimension. As \(n\) grows, the range and distribution of these local types become richer. This produces the observed morphology of \(G_n\): degree landscapes, simplex layers, support strata, support jumps, shells, thickness, anisotropy, rear-central organization, axial structure, spine phenomena, and morphogenesis across \(n\).

On the global side, however, the topology of \(K_n\) is not controlled directly by the largest local simplices. Every simplex of \(K_n\) lies in a full star/top simplex. These canonical simplices form a good cover \(\mathcal C_n\), and hence
\[
K_n\simeq N(\mathcal C_n)=N_n.
\]
The nerve admits a second reduction through anchor simplices and their intersection poset:
\[
N_n\simeq \Delta(J_n).
\]
The anchor-intersection rigidity implies
\[
\dim\Delta(J_n)\le 2.
\]
Thus the high-dimensional local simplices of \(K_n\) are absorbed into contractible containers, while the global homotopy type is governed by a low-dimensional overlap model.

Together with simple connectedness and homology concentration, this gives the global conclusion
\[
K_n\simeq \bigvee^{b_n}S^2,
\qquad
b_n=\chi(K_n)-1.
\]
Thus the qualitative homotopy type is fixed: the remaining global topological information is numerical.

The main conceptual conclusion is therefore:
\[
\boxed{\text{local complexity grows inside canonical containers,}}
\]
whereas
\[
\boxed{\text{global topology is controlled by their low-dimensional overlap pattern.}}
\]
This explains why growing local simplex dimension, richer degree landscapes, support transitions, shell thickening, anisotropy, and other morphological phenomena do not force higher-dimensional global homology. They contribute to the enumeration of cliques, containers, overlaps, and hence to \(\chi(K_n)\) and \(b_n\), but they do not change the qualitative bouquet-of-spheres form.

The synthesis also identifies the remaining residue of the theory. A direct morphological formula for \(\chi(K_n)\), a partition-theoretic interpretation of the bouquet rank \(b_n\), and a precise account of how phase boundaries, the self-conjugate axis, shell structure, and rear-central morphology appear in the nerve or intersection poset remain open. These problems mark the next stage: not the discovery of new qualitative homotopy types, but the interpretation of the numerical and morphological content of the already established global topology.

The conceptual value of the synthesis is therefore twofold. It gives a compact explanation of why the local/global contrast is not paradoxical, and it identifies the remaining work with sharper precision. Future progress should not seek complexity in the qualitative homotopy type itself; the bouquet theorem rules that out. It should instead seek complexity in the rank \(b_n\), in the decomposition of \(\chi(K_n)\), and in the way the morphological regions of \(G_n\) contribute to the low-dimensional overlap model.

\appendix

\section{Status Table for the Main Ingredients}\label{app:status}

\begin{center}
\small
\begin{tabular}{>{\raggedright\arraybackslash}p{0.38\textwidth} >{\raggedright\arraybackslash}p{0.23\textwidth} >{\raggedright\arraybackslash}p{0.25\textwidth}}
\toprule
Statement & Role in this paper & Status \\
\midrule
Local transfer type determines local neighborhood & local foundation & previously proved \\
Unbounded local simplex dimension & contrast witness & strict corollary \\
Star/top containment of cliques & container mechanism & previously proved \\
Canonical cover is good & nerve reduction & previously proved \\
Anchor-intersection rigidity & overlap control & previously proved \\
Anchor closure and intersection-poset reduction & overlap compression & previously proved; proof sketch included \\
\(\dim\Delta(J_n)\le 2\) & low-dimensional model & previously proved; proof sketch included \\
\(K_n\simeq\bigvee^{b_n}S^2\) & global endpoint & previously proved \\
Morphology contributes numerically & synthesis principle & conceptual consequence \\
Direct morphological formula for \(b_n\) & future refinement & open problem \\
\bottomrule
\end{tabular}

\vspace{3pt}
\refstepcounter{table}\label{tab:status}
\textsc{Table~\thetable.} Logical status of the main ingredients used in the synthesis.
\end{center}

\section*{Acknowledgments}

The author used ChatGPT, an AI language model developed by OpenAI, as an auxiliary research and writing assistant during the preparation of this manuscript. Its role was limited to helping organize the synthesis, draft explanatory passages, check internal consistency, and refine terminology across the cycle of papers on partition graphs. The author is solely responsible for the mathematical content, proofs, interpretations, and final form of the article.

\end{document}